\documentclass[10pt,twoside]{amsart}
\usepackage{amssymb,amsmath,amsfonts,amsthm,graphicx,amscd}
\usepackage{enumerate,mathrsfs,longtable,bm,float,}
\usepackage{amsmath,amsfonts,amssymb,amsthm,bm,float}
\usepackage{enumerate}
\usepackage{amscd}
\usepackage[all,cmtip,line]{xy}
\DeclareMathAlphabet{\mathpzc}{OT1}{pzc}{m}{it}

\numberwithin{equation}{section}

\theoremstyle{plain}

\newtheorem{lem}{Lemma}[section]
\newtheorem{lemma}[lem]{Lemma}

\newtheorem{thm}[lem]{Theorem}
\newtheorem{prop}[lem]{Proposition}
\newtheorem{cor}[lem]{Corollary}
\theoremstyle{definition}
\newtheorem{exa}[lem]{Example}
\newtheorem{rem}[lem]{Remark}

\newtheorem{defn}[lem]{Definition}
\newtheorem{Conj}[lem]{Conjecture}

\begin{document}

\baselineskip 16truept

\title{Beck's Conjecture for multiplicative lattices }

\author{Vinayak Joshi and Sachin Sarode}
\address{\rm Department of Mathematics, University of Pune,
Pune-411007, India.} \email{vvj@math.unipune.ac.in \\
sarodemaths@gmail.com}
\date{October 12, 2013}

\maketitle



\begin{abstract}

In this paper, we  introduce the zero divisor graph of a
multiplicative lattice. We provide a counter-example to Beck's
conjecture for multiplicative lattices. Further, we prove that Beck's
conjecture is true for reduced multiplicative lattice which extends
the result of Behboodi and Rakeei \cite{br2} and Aalipour et. al.
\cite{ali}. 

\end{abstract}


\section{Introduction}
In recent years lot of attention have been given to the study of
zero divisor graphs of algebraic structures and ordered
structures. The idea of a zero divisor graph of a commutative ring
with unity was introduced by Beck \cite{B}. He was particularly
interested in the coloring of commutative rings with unity. Many
mathematicians like Anderson et.al. \cite{AN},  F. DeMeyer, T.
McKenzie and K. Schneider \cite{DMS}, Maimani, Pournaki and
Yassemi \cite{mpy}, Redmond \cite{red} and Samei \cite{sam}
investigated the interplay between   properties of the algebraic
structure and graph theoretic properties.

The zero divisor graphs of ordered structures are well studied by
Hala\v{s} and Jukl \cite{hj}, Hala\v{s} and L\"anger \cite{hl},
Joshi \cite{j}, Joshi et.al. \cite{ja, jaa, jwp, jwp1}, Nimbhorkar
et.al \cite{nwl} etc.

In ring theory, the structure of a ring $R$ is closely related to
ideal's behavior more than elements. Hence Behboodi and Rakeei
\cite{br1, br2} introduced the concept of annihilating ideal-graph
$\mathbb{AG}(R)$ of a commutative ring $R$  with unity where the
vertex set $V(\mathbb{AG}(R))$ is the set of non-zero ideals with
non-zero annihilator, that is, for a non-zero  ideal $I$ of  $R$,
$I\in V(\mathbb{AG}(R))$ if and only if there exists a non-zero
ideal $J$ of $R$ such that $I  J=(0)$ and two distinct vertices
$I$ and $J$ are adjacent if and only if $IJ=(0)$ and studied the
properties of rings and its annihilating ideal-graphs. In \cite{br2}, Behboodi and Rakeei
raised the following conjecture.

\begin{Conj} \textit{ For every commutative ring $R$ with
unity, $\chi(\mathbb{AG}(R))= Clique(\mathbb{AG}(R))$.}
\end{Conj}

It is interesting to observe that the set $Id(R)$ of all ideals  of a
commutative ring $R$ with unity forms a modular compactly generated 1-compact  multiplicative lattice in which product of two compact element is compact
(see Definition \ref{1.1}) and the annihilating ideal-graph of a commutative
ring $R$ with unity is nothing but the zero divisor graph of the multiplicative
lattice  of all ideals of $R$ where the vertex set is the set of non-zero zero divisors
and vertices $a$ and $b$ are adjacent if and only if $ab=0$. Hence to study the annihilating ideal-graphs of commutative ring with unity, a multiplicative lattice becomes a tool. This
motivate us to define and study the zero divisor graph of a
multiplicative lattice. It is natural  to ask the following question and the affirmative answer to this question solves Conjecture 1.1. of Behboodi and Rakeei \cite{br2}.

\noindent{\bf Question:} {\it Is Beck's Conjecture, that is,
$\chi(\Gamma_i^m(L))=Clique(\Gamma_i^m(L))$ true for the zero
divisor graphs of a multiplicative lattice $L$ with respect to an
element $i$ of $L$?}\\

In this paper, we  introduce the zero divisor graph of a
multiplicative lattice. We provide a counter-example to Beck's
conjecture for multiplicative lattices, particularly, a
non-reduced multiplicative lattice. Further, we prove that Beck's
conjecture is true for reduced multiplicative lattice which extend
the result of Behboodi and Rakeei \cite{br2} and Aalipour et. al.
\cite{ali}. 

   Now, we begin with necessary concepts and terminology.

\begin{defn} \label{1.1}  A non-empty subset $I$ of a lattice $L$ is said to be \textit{semi-ideal}, if $x \leq a\in I$ implies that $x\in I$. A semi-ideal $I$ of $L$ is said to be an \textit{ideal } if for
$a,b \in I$, $a \vee b \in I$. A proper ideal (semi-ideal) $I$ of
a lattice $L$ is said to be \textit{prime} if $a\wedge b \in I$
implies $a \in I $ or $b \in I$. Dually, we have concept of prime
filter (semi-filter). A prime ideal (semi-ideal) $I$  is a
\textit{minimal prime ideal (semi-ideal)} if there is no prime
ideal (semi-ideal) $Q$ such that $ \{0\} \subsetneqq Q \subsetneqq
I$. A filter is said to be \textit{maximal} if it is a maximal
element of the poset of filter.

For $a\in L$, the set $(a]=\{x \in L~|~ x \leq a\}$ is called the
\textit{principal ideal generated by a}. Dually, we have a concept
of a principal filter $[a)$ generated by $a$.

A lattice $L$ is said to be \textit{complete}, if for any subset
$S$ of $L$, we have $\bigvee S,\bigwedge S\in L$.

A complete lattice $L$ is said to be a \textit{multiplicative
lattice}, if there is defined a binary operation $``\cdot"$ called
multiplication on $L$ satisfying the following conditions:
\begin{enumerate} \item $a \cdot b=b \cdot a$, for all $a,b\in L$, \item
$a \cdot(b \cdot c)=(a \cdot b)\cdot c$, for all $a,b,c\in L$,
\item $a \cdot(\vee_{\alpha} b_{\alpha})=\vee_{\alpha}(a \cdot
b_{\alpha})$, for all $a,b_{\alpha}\in L$, \item $a \cdot b \leq a
\wedge b$ for all $a,b\in L$, \item $a \cdot 1=a$, for all $a\in
L$.
\end{enumerate}

An element $ c $ of a complete lattice $ L $ is said to be
\textit{compact}, if $ c \leq \bigvee_{\alpha} a_{\alpha} $
implies that $ c \leq \bigvee _{i=1}^{n} a_{\alpha_{i}} $, where $
n \in \mathbb{Z}^{+} $. The set of all compact elements of a
lattice $ L $ is denoted by $ L_{*} $. A lattice $L$ is said to be
\textit{compactly generated} or \textit{algebraic}, if for every $
x \in L $, there exist $ x_{\alpha} \in L_{*} $, $\alpha \in
\Lambda$ such that $ x = \vee_{\alpha} x_{\alpha} $, that is,
every element is a join of compact elements.

   A multiplicative lattice $L$ is said to be \textit{$1$-compact } if $1$ is a compact element of $L$. A multiplicative lattice $L$ is said
to be \textit{compact } if every element is a compact element of
$L$.

 An element $p \neq 1$ of a multiplicative lattice $L$ is said to be \textit{prime} if $a
\cdot b \leq p$ implies either $a \leq p$ or $b \leq p$.

Equivalently, an element $p \neq 1$ of a $1$-compact, compactly
generated lattice $L$ is said to be \textit{prime} if $a \cdot b
\leq p$ for $a, b \in L_{*}$ implies either $a \leq p$ or $b \leq
p$.

A nonempty subset $S$ of $L_{*}$ in $1$-compact, compactly
generated lattice is said to be \textit{multiplicatively closed }
if $s_{1},~s_{2} \in S$, then $s_{1} \cdot s_{2} \in S$.

As $L$ is a complete lattice,  it follows that $L$ admits
residuals: for every pair $a, ~b \in L$, there exists an element
$(a:b) = \bigvee \{x \;|\;  x\cdot b \leq a\} \in L$ such that for
any $x \in L$ , $x \cdot b \leq a \Leftrightarrow x \leq (a:b)$.
Clearly, $a \leq (a:b)$ for all $a,~ b \in L$.

 In a multiplicative lattice $L$, an element $a \in L$  is
said to be \textit{nilpotent}, if $a^{n}= 0$, for some $ n \in
\mathbb{Z}^{+}$ and  $L$ is said to be \textit{reduced}, if the
only nilpotent element is 0.

Let $a$ be an element of a multiplicative lattice then we define
$a^*=\bigvee\{x \in L \;|\; a^n \cdot x=0\}$ and if $L$ is
reduced, then $a^*=\bigvee\{x \in L \;|\; x \cdot a=0\}$.

 A lattice $L$ with 0 is said to be \textit{0-distributive} if $a
\wedge b=0=a\wedge c$ then $a\wedge (b \vee c)=0$; see Varlet
\cite{var}. The concept of 0-distributive poset can be found in
\cite{jw, jm}.
\end{defn}

\section{Zero-divisor graph of a multiplicative lattice}

Joshi \cite{j} introduced the zero-divisor graph of a poset with
respect to an ideal $I$. We mentioned this definition, when a poset is a lattice.

\begin{defn}\label{2.1.} Let $I$ be an ideal of a  lattice $L$. We associate an undirected and simple graph,
called the {\it zero-divisor graph of $L$ with respect to $I$},
denoted by $\Gamma_I(L)$ in which the set of vertices is
$\{x\not\in I \;|\; x \wedge y \in I$ for some $y\not\in I\}$ and
two distinct vertices $a,b$ are adjacent if and only if $a \wedge
b \in I$.
\end{defn}

We illustrate this concept with an example.

\begin{exa}\label{2.2} The lattice  $L$ and its zero divisor graph $\Gamma_{\{0\}}(L)$ (in the sense of
Joshi \cite{j}) is shown below.

\begin{figure}[h]
\unitlength .61mm 
\linethickness{0.4pt}
\ifx\plotpoint\undefined\newsavebox{\plotpoint}\fi 
\begin{picture}(108.266,80.581)(15,10)
\put(33.5,56.581){\circle{3.5}} \put(42.54,71.192){\circle{3.5}}
\put(49.372,47.804){\circle{3.5}} \put(43.25,83.581){\circle{3.5}}
\put(33.5,39.331){\circle{3.5}} \put(42.25,30.081){\circle{3.5}}
\put(42.25,73.831){\line(0,-1){.25}}
\put(41,70.331){\line(-2,-3){8}}
\put(42.75,81.831){\line(0,-1){8.5}}
\put(42.75,73.581){\line(0,-1){.5}}
\put(55.25,48.081){\makebox(0,0)[rc]{c}}
\put(42,24.331){\makebox(0,0)[cb]{0}}
\put(37.5,72.331){\makebox(0,0)[lc]{$d$}}
\put(42.75,90.581){\makebox(0,0)[ct]{1}}
\multiput(43.379,31.469)(.047347107,.121181818){121}{\line(0,1){.121181818}}
\put(33.393,54.751){\line(0,-1){13.612}}
\put(33.393,41.192){\line(0,-1){.053}}
\multiput(34.234,37.775)(.048395683,-.047258993){139}{\line(1,0){.048395683}}
\multiput(43.379,69.729)(.0473,-.168175){120}{\line(0,-1){.168175}}
\put(71.812,31.416){\circle{3.784}}
\put(83.164,48.444){\circle{3.784}}
\put(65.4,48.497){\circle{3.784}}
\multiput(65.19,46.658)(.047211864,-.114915254){118}{\line(0,-1){.114915254}}
\multiput(82.218,46.815)(-.04748731,-.070162437){197}{\line(0,-1){.070162437}}
\put(87.947,48.55){\makebox(0,0)[rc]{$b$}}
\put(71.444,26.581){\makebox(0,0)[cb]{$c$}}
\put(60.879,48.55){\makebox(0,0)[lc]{$a$}}
\put(41.444,17.74){\makebox(0,0)[cb]{$L$}}
\put(39.444,7.5){\makebox(0,0)[cb]{$(a)$}}
\put(71.444,17.74){\makebox(0,0)[cb]{$\Gamma_{\{0\}}(L)$}}
\put(72.444,7.5){\makebox(0,0)[cb]{$(b)$}}
\put(121.516,28.749){\circle{3.5}} \put(99.189,48){\circle{3.354}}
\put(121.439,48){\circle{3.354}}
\put(99.689,28.75){\circle{3.354}}
\put(119.939,47.25){\line(0,1){.25}}
\put(98.939,23.5){\makebox(0,0)[cb]{$a$}}
\put(120.939,23.5){\makebox(0,0)[cb]{$b$}}
\put(121.439,54.25){\makebox(0,0)[ct]{$c$}}
\put(98.689,53){\makebox(0,0)[ct]{$d$}}
\multiput(100.387,46.838)(.0556396648,-.0474189944){358}{\line(1,0){.0556396648}}
\put(100.755,48.467){\line(1,0){19.025}}
\put(119.465,48.467){\line(1,0){.368}}
\put(99.126,46.365){\line(0,-1){15.977}}
\put(99.126,30.44){\line(0,-1){.105}}
\put(121.199,30.44){\line(0,1){15.977}}
\multiput(120.148,46.995)(-.0515582656,-.0474281843){369}{\line(-1,0){.0515582656}}
\put(101.386,28.443){\line(1,0){18.5}}
\put(101.491,28.391){\line(-1,0){.158}}
\put(108.939,17.5){\makebox(0,0)[cb]{$\Gamma^m(L)$}}
\put(110.939,7.5){\makebox(0,0)[cb]{$(c)$}}
\put(70.939,.5){\makebox(0,0)[cb]{\bf Figure 2}}
\put(29.5,39.25){\makebox(0,0)[cc]{$a$}}
\put(29.75,57.75){\makebox(0,0)[cc]{$b$}}
\end{picture}
\end{figure}

\end{exa}

Now, we introduced the zero-divisor graph $\Gamma^m(L)$ of a
multiplicative lattice $L$ and illustrate with an example.

\begin{defn}\label{2.1.} Let $L$ be a multiplicative lattice and let $i \in L$. We associate an undirected and simple graph, called the {\it
zero-divisor graph of $L$ with respect to an element i}, denoted by $\Gamma^m_i(L)$ in which the
set of vertices is  $\{x(\nleq i)\in L \;|\; x\cdot y \leq i$ for some
$y(\nleq i)\in L\}$ and two distinct vertices $a,b$ are adjacent if
and only if $a \cdot b \leq i$. Whenever $i=0$, we denote $\Gamma^m_i(L)$ by simply $\Gamma^m(L)$.
\end{defn}

\begin{exa}
Consider the same lattice $L$ shown in Figure 2$(a)$ with the
trivial multiplication $x \cdot y =0=y \cdot x$, for each $x \neq
1 \neq y$ and $x \cdot 1 =x=1 \cdot x$ for every $x \in L$. Then
it is easy to see that $L$ is a multiplicative lattice. Further,
it's zero divisor graph $\Gamma^m(L)$ (in the multiplicative
lattice sense) is shown in Figure 2$(c)$. It is interesting to
note that if 1 is completely join-irreducible({\it i.e.}
$1=\bigvee x_i \Rightarrow 1=x_i$ for some $i$) then any lattice
with this trivial multiplication is a multiplicative lattice.

\end{exa}

\begin{defn}
The \textit{chromatic number} of $G$ is denoted by $\chi(G)$.
Thus, $\chi(G)$ is the minimum number of colors which can be
assigned to the elements of $G$ such that adjacent elements
receive different colors.  A \emph{clique} of a graph $G$ is a
complete subgraph and the supremum of the sizes of clique in $G$,
denoted by $\omega(G)$, is called the \emph{clique number} of $G$.
\end{defn}
For undefined concepts in lattices and graphs, see Gr\"atzer
\cite{Gg} and Harary \cite{H} respectively.

It is known that Beck's Conjecture, that is,
$\chi(\Gamma_I(P))=Clique(\Gamma_I(P))$ is true for the zero
divisor graph of a poset $P$ (with 0) with respect to an ideal $I$
of $P$; see \cite[Theorem 2.9]{j} (when $I=\{0\}$); see also
\cite[Theorem 2.13]{hj}). Hence it is natural to ask the following
question.\\

\noindent{\bf Question:} {\it Is Beck's Conjecture, that is,
$\chi(\Gamma_i^m(L))=Clique(\Gamma_i^m(L))$ true for the zero
divisor graphs of a multiplicative lattice $L$ with respect to an
element $i$ of $L$?}\\

We answer this question negatively in the following example.

\begin{exa}\label{non-red}

Consider the lattice $L$ depicted in the Figure 3$(a)$. Define a
multiplication on $L$ as follows. It is not very difficult to
prove that $L$ is a multiplicative lattice. Moreover, $f^2=0$ for
$f\not=0$ shows that $L$ is non-reduced. Now consider the zero
divisor graph $\Gamma^m(L)$ of $L$ depicted in Figure 3$(b)$. It
is easy to see that $4=\chi(\Gamma^m(L))
>Clique(\Gamma^m(L))=3$. Thus Beck's Conjecture is not true in
the case of multiplicative lattices.

\begin{figure}[h] 
\unitlength .9mm 
\linethickness{0.4pt}
\ifx\plotpoint\undefined\newsavebox{\plotpoint}\fi 
\begin{picture}(105.758,78.947)(0,15)
\put(23.19,16.054){\circle{1.904}}
\put(23.19,42.96){\circle{1.904}}
\put(13.081,42.96){\circle{1.904}}
\put(13.081,58.122){\circle{1.904}}
\put(22.595,58.122){\circle{1.904}}
\put(32.406,58.122){\circle{1.904}}
\put(32.406,43.406){\circle{1.904}}
\put(41.92,58.122){\circle{1.904}}
\put(22.595,86.961){\circle{1.904}}
\put(23.041,25.271){\line(0,-1){8.324}}
\put(23.041,42.068){\line(0,-1){15.014}}
\multiput(23.041,27.054)(-.0374496124,.0587713178){258}{\line(0,1){.0587713178}}
\put(22.595,86.218){\line(0,-1){10.852}}
\multiput(5.5,59.312)(.0374187643,.0377574371){437}{\line(0,1){.0377574371}}
\multiput(23.338,75.217)(.0425217391,-.037416476){437}{\line(1,0){.0425217391}}
\put(22.446,73.879){\line(0,-1){15.162}}
\multiput(21.852,74.177)(-.037482906,-.064799145){234}{\line(0,-1){.064799145}}
\multiput(23.338,74.325)(.037482906,-.065431624){234}{\line(0,-1){.065431624}}
\put(42.068,57.231){\line(0,-1){12.933}}
\multiput(31.96,42.514)(-.037331839,-.069991031){223}{\line(0,-1){.069991031}}
\multiput(41.622,42.663)(-.0407868481,-.037414966){441}{\line(-1,0){.0407868481}}
\put(4.905,57.528){\line(0,-1){13.379}}
\multiput(5.351,42.366)(.0405431235,-.0374242424){429}{\line(1,0){.0405431235}}
\multiput(12.635,57.379)(-.037353846,-.070133333){195}{\line(0,-1){.070133333}}
\multiput(33.149,57.379)(.037333333,-.061086758){219}{\line(0,-1){.061086758}}
\multiput(5.5,57.825)(.0434987277,-.0374452926){393}{\line(1,0){.0434987277}}
\multiput(41.325,57.231)(-.045561039,-.0374545455){385}{\line(-1,0){.045561039}}
\multiput(21.852,57.528)(-.037471074,-.057739669){242}{\line(0,-1){.057739669}}
\put(23.19,57.677){\line(2,-3){8.919}}
\multiput(13.973,57.528)(.0460393701,-.0374566929){381}{\line(1,0){.0460393701}}
\multiput(31.663,57.379)(-.0477108753,-.0374588859){377}{\line(-1,0){.0477108753}}
\multiput(54.406,93.799)(-.037,.037){4}{\line(0,1){.037}}
\put(25.271,14.716){\makebox(0,0)[cc]{\tiny$0$}}
\put(24.973,87.853){\makebox(0,0)[cc]{\tiny$1$}}
\put(26.46,25.568){\makebox(0,0)[cc]{\tiny$f$}}
\put(1.635,42.811){\makebox(0,0)[cc]{\tiny$a$}}
\put(10.406,43.703){\makebox(0,0)[cc]{\tiny$b$}}
\put(19.176,43.703){\makebox(0,0)[cc]{\tiny$c$}}
\put(28.244,43.703){\makebox(0,0)[cc]{\tiny$d$}}
\put(38.501,43.852){\makebox(0,0)[cc]{\tiny$e$}}
\put(4.757,58.717){\circle*{1.904}}
\put(4.905,43.406){\circle*{1.88}}
\put(41.474,43.555){\circle*{1.808}}
\put(22.595,74.92){\circle*{2.081}}
\put(23.19,26.311){\circle*{1.904}}
\put(2.081,60.758){\makebox(0,0)[cc]{\tiny$a\vee c$}}
\put(11,60.204){\makebox(0,0)[cc]{\tiny$a\vee d$}}
\put(19.919,60.501){\makebox(0,0)[cc]{\tiny $b\vee d$}}
\put(28.838,60.055){\makebox(0,0)[cc]{\tiny$b\vee e$}}
\put(37,59.758){\makebox(0,0)[cc]{\tiny$c\vee e$}}
\put(24.527,77.298){\makebox(0,0)[cc]{\tiny$t$}}
\put(23.041,11.108){\makebox(0,0)[cc]{$L$}}
\put(82.947,64.663){\circle*{2.081}}
\put(92.164,56.041){\circle*{1.88}}
\put(92.312,44.893){\circle*{1.88}}
\put(73.879,55.893){\circle*{1.88}}
\put(74.028,44.744){\circle*{1.808}}
\put(83.393,52.771){\circle*{1.808}}
\multiput(82.799,64.812)(-.038086777,-.037471074){242}{\line(-1,0){.038086777}}
\put(73.731,56.041){\line(0,-1){11.595}}
\put(73.731,44.447){\line(1,0){18.581}}
\put(92.312,44.447){\line(0,1){11.892}}
\multiput(92.312,56.339)(-.04136087,.037482609){230}{\line(-1,0){.04136087}}
\multiput(83.096,52.771)(-.042663677,-.037327354){223}{\line(-1,0){.042663677}}
\put(82.799,64.812){\line(0,-1){12.041}}
\multiput(82.799,52.771)(-.10619048,.03716667){84}{\line(-1,0){.10619048}}
\multiput(92.461,56.487)(-.09217,-.03716){100}{\line(-1,0){.09217}}
\multiput(83.244,52.771)(.040730594,-.037333333){219}{\line(1,0){.040730594}}
\multiput(82.799,64.812)(.0374634146,.0380691057){246}{\line(0,1){.0380691057}}
\multiput(92.164,56.487)(.038052,-.03746){250}{\line(1,0){.038052}}
\put(92.164,45.19){\line(0,-1){11.595}}
\put(73.731,44.744){\line(-1,0){11}}
\put(73.879,55.893){\line(0,1){11.595}}
\put(92.312,74.474){\circle*{1.88}}
\put(101.677,46.974){\circle*{2.081}}
\put(92.312,33.744){\circle*{1.784}}
\put(62.879,44.595){\circle*{1.784}}
\put(73.879,67.339){\circle*{1.994}}
\multiput(92.61,74.474)(-.0374674797,-.0876178862){246}{\line(0,-1){.0876178862}}
\multiput(83.393,52.771)(.109577143,-.037377143){175}{\line(1,0){.109577143}}
\multiput(83.393,53.068)(.037331839,-.08532287){223}{\line(0,-1){.08532287}}
\multiput(62.582,44.595)(.09532287,.037331839){223}{\line(1,0){.09532287}}
\multiput(73.879,67.488)(.037478992,-.059962185){238}{\line(0,-1){.059962185}}
\put(83.096,52.92){\line(0,-1){18.135}}
\put(82.947,34.041){\circle*{1.994}}
\put(82.204,49.798){\makebox(0,0)[cc]{\tiny$f$}}
\put(84.136,20.811){\makebox(0,0)[cc]{$\Gamma^m(L)$}}
\put(82.136,15.811){\makebox(0,0)[cc]{\tiny (R)-Red,
(B)-Blue,(G)-Green, (W)-White}}
\put(87.623,56.001){\makebox(0,0)[cc]{\tiny(B) }}
\put(82.204,67.339){\makebox(0,0)[cc]{\tiny $a$ (R)}}
\put(95.353,58.717){\makebox(0,0)[cc]{\tiny$b$ (W)}}
\put(96.096,42.663){\makebox(0,0)[cc]{\tiny$ c$(R)}}
\put(71.947,41.92){\makebox(0,0)[cc]{\tiny$d$(W)}}
\put(70.461,57.528){\makebox(0,0)[cc]{\tiny$ e$ (G)}}
\put(92.758,77.744){\makebox(0,0)[cc]{\tiny$b\vee e$(W)}}
\put(105.758,49.352){\makebox(0,0)[cc]{\tiny$ a\vee c$(R)}}
\put(98.326,35.974){\makebox(0,0)[cc]{\tiny$ b\vee d$(W)}}
\put(80.271,32.406){\makebox(0,0)[cc]{\tiny$ t$(R)}}
\put(73.285,70.163){\makebox(0,0)[cc]{\tiny$ a\vee d$(R)}}
\put(61.987,47.568){\makebox(0,0)[cc]{\tiny$ c\vee e$(R)}}
\put(23.582,06.707){\makebox(0,0)[cc]{$(a)$}}
\put(83.664,06.801){\makebox(0,0)[cc]{$(b)$}}
\put(58.513,01.884){\makebox(0,0)[cc]{$\textbf{Figure ~ 3}$}}
\end{picture}

\end{figure}

\begin{tabular}{|c|c|c|c|c|c|c|c|c|c|c|c|c|c|c|}
  \hline
  $\bullet$& 0 & $a$ &$ b$ & $c$ & $d$ & $e$ & $f$ & $(a\vee c)$ & $(a \vee d)$& $(b \vee e)$ & $(c \vee e)$ & $(b \vee d)$ &$ t$ & $1$ \\
  \hline
  0& 0 & 0 & 0 & 0 & 0 & 0 & 0& 0 & 0 & 0 & 0 & 0 & 0 & 0 \\ \hline
 $ a$ & 0 & $f$ & 0 & $f$ & $f$ & 0 & 0 & $f$ & $f$ & 0 & $f$ & $f$ & $f$ & $a$
 \\  \hline
  $b$ & 0 & 0 & $f$ & 0 & $f$ & $f$ & 0 & 0 & $f$ & $f$ & $f$ & $f$ & $f$ & $b$
  \\  \hline
  $c$& 0 & $f$ & 0 & $f$ & 0 & $f$ & 0 & $f$ & $f$ & $f$ & $f$ & 0 & $f$ & $c$
  \\  \hline
  $d$ & 0 & $f$ & $f$ & 0 & $f$ & 0 & 0 & $f$ & $f$ & $f$ & 0 & $f$ & $f$ & $d$
  \\  \hline
  $e$ & 0 & 0 & $f$ & $f$ & 0 & $f$ & 0 & $f$ & 0 & $f$ & $f$ & $f$ & $f$ & $e$
  \\  \hline
  $f$ & 0 & 0 & 0 & 0 & 0 & 0 & 0 & 0 & 0 & 0 & 0 & 0 & 0 & $f$ \\
    \hline
  $(a \vee c)$ & 0 & $f$ & 0 & $f$ & $f$ & $f$ & 0 & $f$ & $f$ & $f$ & $f$ & $f$ & $f$ & $(a \vee c)$
  \\  \hline
  $(a \vee d)$ & 0 & $f$ & $f$ & $f$ & $f$ & 0 & 0 & $f$ & $f$ & $f$ & $f$ & $f$ & $f$ & $(a \vee d)$
  \\  \hline
  $(b \vee e)$ & 0 & 0 & $f$ & $f$ & $f$ & $f$ & 0 & $f$ & $f$ & $f$ & $f$ & $f$ & $f$ & $(b \vee e)$
  \\  \hline
  $(c\vee e)$ & 0 & $f$ & $f$ & $f$ & 0 & $f$ & 0 & $f$ & $f$ & $f$ & $f$ & $f$ & $f$ & $(c \vee e)$
  \\  \hline
  $(b \vee d)$ & 0 & $f$ & $f$ & 0 & $f$ & $f$ & 0 & $f$ & $f$ & $f$ & $f$ & $f$ & $f$ & $(b \vee d)$
  \\  \hline
  $t$ & 0 & $f$ & $f$ & $f$ & $f$ & $f$ & 0 & $f$ & $f$ & $f$ & $f$ & $f$ & $f$ &
  t\\  \hline
  $1$ & 0 & $a$ & $b$ & $c$ & $d$ & $e$ & $f$ & $(a\vee c)$ & $(a \vee d)$ & $(b \vee e)$ & $(c \vee e)$ & $(b \vee d)$ & $t$ & $1$
  \\  \hline

\end{tabular}\\

\end{exa}

\begin{rem}\label{red} If $R$ is a commutative ring with unity, then it is well known that the ideal lattice $Id(R)$ of $R$ is 1-compact, compactly generated modular
multiplicative lattice; see Anderson \cite{ddand}. Further, it is
easy to observe that if $R$ is reduced then $Id(R)$ is a reduced
multiplicative lattice. The lattice depicted in Figure $3(a)$ is a
non-modular lattice (as it contains a non-modular sublattice shown
in dark black circles) and hence it can not be an ideal lattice of
any commutative ring unity. Therefore the above conjecture remains
open though Beck's conjecture fails in the case of non-reduced
multiplicative lattices. We have more pleasant situation when a
multiplicative lattice is reduced. For this, we need  Theorem 2.9
of \cite{j}. Note that the notion of prime semi-ideals mentioned
in \cite{j} coincides with the corresponding notions in lattices.
Hence we quote essentially Theorem 2.9 of \cite{j}, when a poset
is a lattice and an ideal $I=\{0\}$.\end{rem}

\begin{thm}[Joshi \cite{j}] \label{2.9} Let $L$ be a lattice.  If $Clique(\Gamma(L)) < \infty$ then $L$ has a
finite number of minimal prime semi-ideals and if $n$ is this number
then $\chi(\Gamma(L)) =  Clique(\Gamma(L))  = n$.\end{thm}

\begin{lemma}\label{0-dist} Let $L$ be a reduced multiplicative lattice. Then $L$ is 0-distributive.\end{lemma}

\begin{proof} Let $a \wedge b=0=a \wedge c$ for $a,b,c \in L$. Since $a\cdot b \leq a \wedge b$ and $L$
is a multiplicative lattice, we have $a\cdot(b \vee c)=0$. Further,  $L$ is reduced, we have $a \cdot b=0$
 implies $a \wedge b=0$. This together with $a\cdot(b \vee c)=0$ proves that $L$ is 0-distributive.\end{proof}

It is proved in Joshi and Mundlik \cite{jm} that every prime
semi-ideal in a 0-distributive poset is a prime ideal of a
0-distributive poset. But for the sake of completeness, we provide
the proof of the same in the following result which is essential
for the proof of the Beck's Conjecture.

\begin{thm}\label{beck} Let $L$ be a reduced multiplicative lattice and let $Clique(\Gamma^m(L)) < \infty$.
Then $L$ has a finite number of minimal prime ideals and if $n$ is
this number then Beck's Conjecture is true, that is,
$\chi(\Gamma^m(L))=Clique(\Gamma^m(L))=n$. \end{thm}

\begin{proof} Suppose $L$ is a reduced multiplicative lattice. Then one can
easily prove that whenever $a\cdot b=0$ then $a\wedge b=0$ and
conversely for $a, b\in L$.  By Lemma \ref{0-dist}, $L$ is
0-distributive.

Now, we prove that every minimal prime semi-ideal of $L$ is a
minimal prime ideal of $L$. Let $I$ be a minimal prime semi-ideal
of $L$. To prove $I$ is an ideal, it is enough to show that for
$a, b \in I$, $a\vee b \in I$. Let $a, b \in I$. Since $I$ is a
minimal prime semi-ideal of $L$, it is easy to observe that $L
\setminus I$ is a maximal filter of $L$. Further, $a, b\not\in L
\setminus I$, we have $[a) \vee (L\setminus I)= [b) \vee
(L\setminus I)=L=[0)$. Hence there exists $t \in L\setminus I$
such that $t \wedge a=0=t \wedge b$. By Lemma \ref{0-dist}, we
have $ t \wedge (a \vee b)=0$. This proves that $a\vee b \in  I$,
otherwise $0=t \wedge (a\vee b) \in L \setminus I$, a
contradiction to maximality of $L \setminus I$. This proves that
every minimal prime semi-ideal is a minimal prime ideal.

In view of the observation whenever $a\cdot b=0$ then $a\wedge
b=0$ and conversely for $a, b\in L$, the zero divisor graph
$\Gamma(L)$ of the lattice $L$ (in the lattice sense) is
isomorphic to the zero divisor graph $\Gamma^m(L)$ of the reduced
multiplicative lattice $L$. Since $Clique(\Gamma^m(L)) < \infty$,
we have $Clique(\Gamma(L)) < \infty$. Hence by Theorem \ref{2.9},
Beck's Conjecture is true for general lattices (in fact for
posets);  hence it is true for reduced multiplicative lattice,
that is, $\chi(\Gamma^m(L))=Clique(\Gamma^m(L))=n$, where $n$ is
the number of minimal prime ideals of $L$.
\end{proof}

\begin{rem}\label{red1}
It is obvious that the prime ideals in a commutative ring $R$ with
unity are nothing but the prime elements of the $Id(R)$. In view
of this observation and the fact that the annihilating ideal-graph
$\mathbb{AG}(R)$ of a commutative ring $R$ with 1 is nothing but
the zero divisor graph of a multiplicative lattice $Id(R)$  of all
ideals of a commutative ring $R$ with 1, Theorem \ref{beck} extend
Corollary 2.11  of Behboodi and Rakeei \cite{br2} but not
completely Theorem 8 of Aalipour et. al. \cite{ali}. In order to
extend Theorem 8 of Aalipour et. al. \cite{ali}, we have to prove
that $\chi(\Gamma^m(L))=Clique(\Gamma^m(L))=n$, where $n$ is the
number of minimal prime elements of a reduced multiplicative
lattice $L$. We achieve this result in sequel. Before proceeding
further, we provide an example of a reduced multiplicative lattice
which has prime ideals but not have any prime element.  It should
be noted that a reduced multiplicative lattice always has a prime
ideal but need not have a prime element. Let $\mathbb{N}$ be the
set of natural numbers. Let $L=\{X \subseteq \mathbb{N}~|~ |X| <
\infty\}\cup \{\mathbb{N}\}$. Then it is easy to see that $L$ is a
reduced multiplicative lattice with multiplication as the meet.
One can prove that the set $\{n\}^\bot=\{A \subseteq \mathbb{N}~|~
A \cap \{n\}=\emptyset \textrm{ and } |A|< \infty\}$ is a minimal
prime ideal of $L$ for every $n \in \mathbb{N}$. But $L$ does not
contain any prime element.

\end{rem}

\begin{lem}\label{6.1.} Let $L$ be reduced 1-compact, compactly generated lattice and
$ x \in L$ if $x^{*}$ is maximal among $\{a^{*} \;|\; a \in L,
a^{*} \neq 1 \}$, then $x^{*}$ is a prime element of $L$\end{lem}

\begin{proof} Suppose that $a \cdot b \leq x^{*}$ and $a \nleq
x^{*}$. Then $(x \cdot a \cdot b)=0$. Let $(0 \neq y) \leq (x
\cdot a)$. Then $(b \cdot y) \leq (x \cdot a \cdot b)=0$. Thus $b
\leq y^{*}$. As $y \leq x$ implies $x^{*} \leq y^{*}$ and $y^{*}
\neq 1$ due to $y \neq 0$. By maximality of $x^{*}$, we deduced
that $y^{*}=x^{*}$, hence $ b \leq x^{*}$. This proves that
$x^{*}$ is prime.
\end{proof}

\begin{lem}\label{6.2.} Let $L$ be reduced $1$-compact, compactly generated lattice. If $x^{*}, y^{*}$ are distinct prime
elements of $L$, then $x \cdot y=0$.\end{lem}

\begin{proof} Assume contrary that $x \cdot y \neq 0$, that is $x \nleq
y^{*}$ and $y \nleq x^{*}$. Consider a compact element $t \leq
x^{*}$. Then $x \cdot t=0$. As $y^{*}$ is prime and $x \cdot t
\leq y^{*}$. Since $L$ is compactly generated and every compact
element $t \leq x^{*}$ is also $\leq y^{*}$, we have $t \leq
y^{*}$ and hence $x^{*} \leq y^{*}$. Similarly we can show $y^{*}
\leq x^{*}$. Hence $y^{*}=x^{*}$, a contradiction.\end{proof}

\begin{lem}\label{6.3.} Let $L$ be reduced $1$-compact, compactly generated lattice  with $Clique(\Gamma^m(L)) < \infty$,
then the set $\{x^{*} \;|\; x \in L, x \neq 0\}$ satisfies the
ascending chain condition.\end{lem}

\begin{proof} Suppose $a_{1}^{*} < a_{2}^{*} < a_{3}^{*} < a_{4}^{*} \cdot \cdot \cdot
\cdot$. Let $x_{j} \leq a_{j}^{*}$ and  $x_{j} \nleq a_{j-1}^{*}$,
$j=2, 3, \cdot \cdot \cdot$. If we let $y_{n}=(x_{n} \cdot
a_{n-1})$, $n=2, 3, \cdot \cdot \cdot$, then $y_{n} \neq 0$. For
$i < j$, we have $x_{i} \leq a_{i}^{*} \leq a_{j-1}^{*}$. Thus
$(x_{i} \cdot a_{j-1})=0$, consequently, $(y_{i} \cdot y_{j})=0$
for all $i \neq j$. Thus the set $\{y_{n} \;|\; n=2, 3, \cdot
\cdot \cdot \}$ is an infinite clique, a contradiction.\end{proof}

\begin{lem}\label{6.4.} Let $L$ be reduced $1$-compact, compactly generated lattice  with $Clique(\Gamma^m(L))< \infty$,
then the set of all distinct maximal annihilator elements of $L$
is finite.\end{lem}

\begin{proof} Let $A=\{x_{i}^{*} \;|\; x_{i}^{*}~ is~maximal\}$ be the set
of all maximal annihilator elements of $L$. Clearly, $x_{i} \neq
0$ for all $i$ and $x_{i}^{*} \neq x_{j}^{*}$ whenever $i \neq j$.
By Lemma \ref{6.1.}, all the elements of $A$ are prime. Then by
Lemma \ref{6.2.}, $x_{i} \cdot x_{j}=0$ for all $i \neq j$. This
shows $Clique(L) \geq |A|$, which according to $Clique(L) <
\infty$ yields the finiteness of $A$.\end{proof}

\begin{lem}\label{6.5.} Let $L$ be reduced 1-compact, compactly generated lattice with $L_{*}$ multiplicatively closed set and  $Clique(\Gamma^m(L))
< \infty$, then $0$ is the meet of a finite number of minimal
prime elements of $L$.\end{lem}

\begin{proof} According to Lemma \ref{6.4.}, let $\{x_{i}^{*} \;|\; 1 \leq i \leq
n\}$ be the set of all maximal annihilator elements of $L$. By
Lemma \ref{6.1.}, all these elements are prime. Further, due to
Lemma \ref{6.2.}, $x_{i} \cdot x_{j}=0$ for all $i \neq j$. Assume
that there is $(0 \neq) a \leq \bigwedge_{1 \leq i \leq
n}x_{i}^{*}$. Then $a \cdot x_{i}=0$ for all $i$. Thus $x_{i} \leq
a^{*}$ for all $i$. But by Lemma \ref{6.3.}, $a^{*} \leq
x_{i}^{*}$ for some $i$. However, this gives $x_{i} \leq a^{*}
\leq x_{i}^{*}$, that is, $x_{i}^{2}=0$, a contradiction to
reduced lattice. Thus we have, $0=\bigwedge_{1 \leq i \leq
n}x_{i}^{*}$.

     Now, we show that $x_{i}^{*}=p_{i}$ are minimal prime elements
of $L$. Since $p_{i}$ are assumed to be maximal annihilator
elements, we may suppose that none of $p_{i}$ contains $p_{j}$ for
all $i \neq j$. Indeed, if $p_{j}$ were not minimal for some $j$,
then there exists a minimal prime element $q$ with $q < p_{j}$.
Now $\bigwedge_{i=1}^{n}x_{i}^{*}\leq q$ and $q$ is prime implies
that $x_{i}^{*}\leq q$ for some $i$. But then $x_{i}^{*}=p_{i}\leq
q\leq p_{j}$, a contradiction. Thus $\bigwedge^{n}_{i = 1}
p_{i}=0$.

\end{proof}

\begin{lem}\label{6.6.} Let $L$ be reduced 1-compact, compactly generated lattice
with with $L_{*}$ multiplicatively closed set and
$Clique(\Gamma^m(L)) < \infty$, then every minimal prime element
$p$ of $L$ is of the form $x^{*}$ for some $x \in L$.\end{lem}

\begin{proof} Let $p$ be a minimal prime element of $L$. For $x \nleq p$, we
have $x^{*} \leq p$. By Lemma \ref{6.3.}, there are maximal
annihilator elements among $A=\{x^{*} \;|\; x \nleq p\}$. In fact,
we prove that there is a greatest one. Let $y_{1}^{*}, y_{2}^{*}$
be two maximal elements of $A$. We have $y_{1} \cdot y_{2} \nleq
p$, since $p$ is prime and $y_{1}, y_{2} \nleq p$. Thus there is
$(0 \neq y)=(y_{1} \cdot y_{2})$ with $y \nleq p$. Clearly,
$y_{1}^{*}, y_{2}^{*} \leq y^{*}$ and as both $y_{1}^{*},
y_{2}^{*}$ are maximal in $A$, we conclude that $y_{1}^{*}=
y_{2}^{*}=y^{*}$. This shows that $A$ has the greatest element say
$z^{*}$.

       By Lemma \ref{6.1.}, $z^{*}$ is prime. We prove that
$z^{*} \leq p$. If not then there is a compact element $g \leq
z^{*}$ such that $g \nleq p$. Then $z\leq z^{**} \leq g^{*} \in
A$. Hence we have $z \leq g^{*}\leq z^{*}$, a contradiction to the
fact that $L$ is reduced. Thus $z^{*} \leq p$ and $p$ is a minimal
prime element of $L$, we have $p=z^{*}$.

\end{proof}

   Now, we prove Beck's conjecture for reduced lattice $L$.

\begin{thm}\label{beck1} Let $L$ be a reduced 1-compact, compactly generated  lattice
with with $L_{*}$ multiplicatively closed set and
$Clique(\Gamma^m(L)) < \infty$. Then the number of minimal prime
elements of $L$ is finite, say $n$ and
$\chi(\Gamma^m(L))=Clique(\Gamma^m(L))=n$\end{thm}

\begin{proof} By the Lemma \ref{6.5.}, we have $0=\wedge_{1 \leq i \leq
n}p_{i}$ for $p_{i}$ being minimal prime elements of $L$ and by
Lemma \ref{6.6.}, $p_{i}=x_{i}^{*}$ for some $x_{i} \neq 0$, hence
$0=\wedge_{1 \leq i \leq n}x_{i}^{*}$. By Lemma \ref{6.2.},
$\{x_{i} \;|\; 1 \leq i \leq n\}$ is a clique in $L$ and thus
$Clique(L) \geq n$. Define a coloring $\mathscr{f}$ of $L$ as
$f(x)=min\{i \;|\; x \nleq p_{i}\}$. If $x, y$ are adjacent
vertices, then $x \cdot y=0$. If $f(x)=k+1$, then $x \leq p_{i}$
for $1 \leq i \leq k$ and $x \nleq p_{k+1}$. So we conclude that
$y \leq p_{k+1}$, since $x \cdot y =0 \leq p_{k+1}$ and $x \nleq
p_{k+1}$. This show that $f(y) \neq k+1$ and thus $f(x) \neq
f(y)$. Hence $f$ is coloring of $L$. This yields
$\chi(\Gamma^m(L)) \leq n$, and finally $ n \leq
Clique(\Gamma^m(L)) \leq \chi(\Gamma^m(L)) \leq n$. In conclusion,
we have $\chi(\Gamma^m(L))=Clique(\Gamma^m(L))=n$.\end{proof}

\begin{cor} Let $R$ be a reduced commutative ring with unity such that $Clique(\mathbb{AG}(R))< \infty$.
Then $\chi(\mathbb{AG}(R))= Clique(\mathbb{AG}(R))=|Min(R)|$.
\end{cor}

\end{document}